\documentclass[12pt]{amsart}
\usepackage{xcolor}

\usepackage{cite}

\usepackage{graphicx}

\usepackage{amsmath,amsthm,amssymb} 
\newtheorem{theorem}{Theorem}

\newtheorem{conjecture}{Conjecture}
\newtheorem{problem}{Problem}
\usepackage{marginnote,fullpage}

\newcommand{\abs}[1]{\left\lvert #1 \right\rvert}
\newcommand{\pa}[1]{\left( #1 \right)}

\begin{document}

\title{An extremal problem for odd univalent polynomials}

\author{Dmitriy Dmitrishin}
\address{Odessa National Polytechnic University, 1 Shevchenko Ave.,
Odesa 65044, Ukraine; email: {dmitrishin@op.edu.ua }}

\author{Daniel Gray}
\address{Georgia Southern University, Statesboro GA, 30460; \newline email: {dagray@georgiasouthern.edu}}

\author{Alex Stokolos}
\address{Georgia Southern University, Statesboro GA, 30460; \newline email: {astokolos@georgiasouthern.edu}}

\author{Iryna Tarasenko}
\address{Odessa National Academy of Communications,   22/1 Staroportofrankivska Str., Odesa 65000, Ukraine}

\begin{abstract}
For the univalent polynomials $F(z) = \sum\limits_{j=1}^{N} a_j z^{2j-1}$ with real coefficients and  normalization \(a_1 = 1\) we solve the  extremal problem
\[
 \min_{a_j:\,a_1=1} \left( -iF(i) \right) = \min_{a_j:\,a_1=1} \sum\limits_{j=1}^{N} {(-1)^{j+1} a_j}.
\]
We show that the solution is  $\frac12 \sec^2{\frac{\pi}{2N+2}},$ and the extremal polynomial
\[
 \sum_{j = 1}^N \frac{U'_{2(N-j+1)} \left( \cos\pa{\frac{\pi}{2N+2}}\right)}{U'_{2N} \left( \cos\pa{\frac{\pi}{2N+2}}\right)}z^{2j-1}
\]
is unique and univalent, where the \(U_j(x)\) are the Chebyshev polynomials of the second kind and \(U'_j(x)\) denotes the derivative. As an application, we obtain the estimate of the Koebe radius for the odd univalent polynomials in $\mathbb D$ and formulate  several conjectures.

\end{abstract}
\maketitle

\section{{Motivation}}

This work is motivated by problems of classical geometric complex analysis. Let us recall the main milestones 
in the development of the theory.

The paper by Ludwig Bieberbach of 1916, in which he proved the exact estimate $\abs{a_2} \le 2$ for the second coefficient 
of the univalent in the unit disc functions $f (z)$ with the normalization $f (0) = 0$, $f '(0) = 1,$ can be considered 
the fundamental work. Such functions are called schlicht functions. This estimate immediately implies the Koebe 1/4 theorem, 
i.e. the statement that the image of any schlicht function contains a disc of radius 1/4. The assertion was put forward 
as a hypothesis by Paul Koebe, and the function $K(z) =\frac{z}{(1-z)^2} = z + 2z^2 + \ldots + jz^j + \ldots,$ 
now called the Koebe function, is a schlicht function and has the property that $K(-1) = -\frac14,$ thus showing sharpness of the constant. 

In that same work, Bieberbach stated in a footnote the remark that, in general $\abs{a_j} \le j$ is valid. This innocent-looking statement became known as the Bieberbach conjecture, which, since then, 
has been the engine of the development of geometric complex analysis for hundred years.

In 1984, the Bieberbach conjecture was proved by Louis de Branges, who also showed that the extremal function, up to rotation, 
is the Koebe function.

The odd Koebe function
$$
K^{(2)}(z)=(K(z^2))^{\frac12}=\frac{z}{1-z^2} = z+z^3+z^5+\ldots+z^{2j+1}+\ldots 
$$
has generated its own range of tasks, even more intriguing and unexpected than the general case. Thus, J. Littlewood 
and R.E.A.C. Paley conjectured \cite{L-P} that all the coefficients of an odd, i.e. 2-symmetric, schlicht function are bounded in 
absolute value by 1. The conjecture turned out to be false: M. Fekete and G. Szeg\"o proved the exact inequality 
$\abs{a_5} \le 1/2 + e^{-2/3} = 1.013\ldots$ \cite{Fekete}. However, W. Hayman \cite{Hayman} showed that for any odd schlicht function, 
except for the Koebe function $F^{(2)}(z)$, the limit of the absolute value of the coefficients is less than one. This implies 
that the Littlewood-Paley conjecture is valid for all but a finite number of coefficients. 

Many of the above-mentioned problems have the form of extremal problems and can be formulated for the univalent polynomials. Note that the coefficients of univalent functions can be sufficiently large, while the leading coefficient of a schlicht polynomial of degree $N$ is at most $1/N.$
As a result, partial sums of univalent functions, usually, are not univalent polynomials. And in general, the number of interesting examples of univalent polynomials is quite limited.\\

For instance, the value of the Koebe function $K(-1)$ is extreme for the problem of covering a disc by the image of the
unit disc. M.~Brandt~\cite[(486)]{Brandt1987} found the solution to the problem for the polynomials $F(z)=z+a_2z^2+...+a_Nz^N$ that are univalent in the disc $\mathbb D$.
\begin{equation}\label{brn}
\sup_{a_j:\, a_1=1} \left(F(-1)\right) =-\frac14 \sec^2{\frac{\pi}{N+2}}.
\end{equation}
Moreover, he wrote out the extremal polynomial \cite[(390)-(392)]{Brandt1987}. In \cite{DSS}, 
a more general formulation is considered; the extremal polynomial has the form
\begin{equation}\label{extr1}
\sum_{j = 1}^N 
\frac{U'_{N - j + 1}\left(\cos\pa{ \frac{\pi}{N + 2}}\right)}{U'_N\left(\cos\pa{\frac{\pi}{N + 2}}\right)}
 U_{j-1}\left(\cos\pa{\frac{\pi}{N + 2}}\right) z^j,
\end{equation}
which is more convenient for 
further generalizations, and it is shown that the extremal polynomial is unique. This problem can be considered as the problem of covering
a segment on the real axis by the image of the unit disc under the polynomial schlicht mapping. It is interesting to note 
that although the problems of covering a disc and a segment are related, neither of them is a consequence of the other. 
To the curious reader, we recommend to read~\cite{DDS}.

\section{{Introduction}}
Let us now consider an analogous problem for the univalent odd polynomials.

Optimization problem: on the class of odd polynomials of degree $2N-1$ 
$$
F(z) = \sum\limits_{j=1}^{N} {a_j z^{2j-1}} 
$$
that are univalent in $\mathbb D = \left\{ z:\, \abs{z}<1 \right\}$ and have real coefficients, find
\begin{equation}\label{jn}
J_N = \min_{a_j:\,a_1=1} \left( -iF(i) \right) = \min_{a_j:\,a_1=1} \left( \sum\limits_{j=1}^{N} {(-1)^{j+1} a_j} \right),
\end{equation}
as well as the polynomial (extremizer) that solves the optimization problem. Above $i^2=-1.$

Note that the possibility of writing ``minimum'' instead of ``infimum'' is guaranteed by DeBranges theorem, i.e., we can
consider the infimum only on a compact set since the coefficients of schlicht functions do not escape to infinity.

\section{{Main results}}

Consider this optimization problem for a wider class of polynomials than univalent ones, namely for those polynomials $F(z)$ that map 
the semidisc $\hat{D}=\left\{ z\in\mathbb{C}: \, \abs{z}<1,\, \text{Re}\{z\}>0 \right\}$ to the right half-plane 
$\hat{C}=\left\{ z\in\mathbb{C}: \, \text{Re}\{z\}>0 \right\}.$ 

Note that the odd algebraic polynomial $F(z)$ satisfies the previous condition if and only if the polynomial $i F(-iz)$ is typically real\footnote{
A polynomial $f(z)$ is typically real if it is real in the interval \((-1,1)\) and in all other points of the disc $\mathbb D$ we have $\Im\{f(z)\}\Im\{z\}>0$. Every univalent polynomial with real coefficients is also typically real.}

\begin{theorem}\label{01}
If the polynomial $i F(-iz)$ is typically real, then the equality 
$$
J_N = \frac12 \sec^2{\frac{\pi}{2N+2}} = -iF^{(0)}(i) \leq -iF(i)
$$
holds. In this case, the extremizer is unique and is given by the formula 
\begin{equation}\label{extr2}
F^{(0)}(z) = \sum\limits_{j=1}^{N}
{\frac{U'_{2(N-j+1)} \left( \cos\pa{\frac{\pi}{2N+2}}\right)}{U'_{2N} \left( \cos\pa{\frac{\pi}{2N+2}}\right)} z^{2j-1}}.
\end{equation} 
\end{theorem}

\begin{proof}
Since $iF(-iz)$ is typically real, this means $F(z)$ maps the semidisc $\hat{D}$ to the half-plane $\hat{C}$. It follows that $\displaystyle \text{Re}(F(z)) = \sum\limits_{j=1}^{N} {a_j \cos\pa{(2j-1)t}} \ge 0$. Thus, we have 
$$
J_N = \min_{a_j:\,a_1=1}\left( \abs{F(i)} \right) = 
$$
$$
= \inf_{a_j} \left\{ 
1 - a_2 + a_3 - \ldots :
\sum\limits_{j=1}^{N} {a_j \cos\pa{(2j-1)t}} \ge 0, \, t\in \left[ 0, \frac{\pi}{2} \right]
\right\}.
$$
The equality 
\begin{equation}\label{sum}
\sum\limits_{j=1}^{N} {a_j \cos\pa{(2j-1)t}} = \gamma_1 \cos\pa{t} 
\left( 1 + 2\sum\limits_{j=2}^{N} {\frac{\gamma_j}{\gamma_1} \cos\pa{2(j-1)t}} \right)
\end{equation}
holds, where the coefficients $a_1,\ldots,a_N$ and $\gamma_1,\ldots,\gamma_N$ are related by the bijective relationship
\begin{equation}\label{cg}
\left\{\begin{array}{ll}
\gamma_s = \sum\limits_{j=s}^{N} {(-1)^{s+j} a_j}, \\
s = 1,\ldots, N,
\end{array}\right.
\end{equation}
which is equivalent to 
\begin{equation}\label{ca}
\left\{\begin{array}{ll}
a_s = \gamma_s + \gamma_{s+1}, \\
s = 1,\ldots, N. 
\end{array}\right.
\end{equation}

In (\ref{ca}), for convenience, it is assumed that $\gamma_{N+1}=0.$ It follows from (\ref{cg}) or (\ref{ca}) that 
$\gamma_1 = \sum\limits_{j=1}^{N} {(-1)^{s+1} a_j}$ and $a_1 = \gamma_1 + \gamma_2 = 1.$ Then 
$$
J_N = \min_{\gamma_j:\,\gamma_1 + \gamma_2 = 1}
\left\{
\gamma_1:\,
1 + 2\sum\limits_{j=1}^{N-1} {\frac{\gamma_{j+1}}{\gamma_1} \cos\pa{2t j}} \ge 0
\right\}.
$$
According to Problem 52 in \cite[p. 79]{Pol-Sz} or (1.3) in \cite{And-Dim}, a nonnegative trigonometric polynomial of the form
$$
1 + \sum_{j = 1}^{N - 1}\pa{\lambda_j\cos\pa{2t j} + \mu_j\sin\pa{2t j}}
$$
must satisfy the inequality $\sqrt{\lambda_1^2 + \mu_1^2} \leq 2\cos\pa{\frac{\pi}{N + 1}}$. Since our nonnegative trigonometric polynomial
$$
1 + 2\sum\limits_{j=1}^{N-1} {\frac{\gamma_{j+1}}{\gamma_1} \cos\pa{2t j}}
$$
has no sine terms, the inequality simplifies, yielding 
$$
\sqrt{\pa{2 \cdot \frac{\gamma_2}{\gamma_1}}^2} \leq 2\cos\pa{\frac{\pi}{N + 1}} \implies \abs{\frac{\gamma_2}{\gamma_1}} = \abs{\frac{1 - \gamma_1}{\gamma_1}} \le \cos\pa{\frac{\pi}{N+1}}.
$$

From this, taking into account that $0< \gamma_1 <1,$ we obtain
$$
\gamma_1 \ge
\frac{1}{1+\cos\pa{\frac{\pi}{N+1}}} = \frac{1}{2\cos^2\pa{\frac{\pi}{2N+2}}}.
$$
Therefore, $\gamma_1 = \frac{1}{1+\cos\pa{\pi / (N+1)}} = \frac{1}{2\cos^2\pa{\pi / (2N+2)}}.$ The remaining 
coefficients are uniquely determined \cite{Eger} from the condition
\begin{gather*}
\abs{\frac{\gamma_{j+1}}{\gamma_1}} = 
\frac{1}{2(N+1) \sin\pa{\frac{\pi}{N+1}}}
\left[
(N-j+2) \sin\pa{\frac{\pi (j+1)}{N+1}} - \right. 
\\ \left.- (N-j) \sin\pa{\frac{\pi (j-1)}{N+1}}
\right], \\
j=2,\ldots,N-1.
\end{gather*}
Determine from (\ref{ca}) the coefficients of the extremizer (these are unique):
\begin{gather*}
a^0_j = \frac{1}{2(N+1) \sin\pa{\frac{\pi}{N+1}} \left( 1+\cos\pa{\frac{\pi}{N+1}} \right)}
\left( (N-j+3) \sin\pa{\frac{\pi j}{N+1}} - \right.\\
 - (N-j+1) \sin\pa{\frac{\pi (j-2)}{N+1}} + (N-j+2) \sin\pa{\frac{\pi (j+1)}{N+1}} - \\
 \left.- (N-j) \sin\pa{\frac{\pi (j-1)}{N+1}} \right), \\
j=1,\ldots,N-1,
\end{gather*}
\begin{gather*}
a^{(0)}_{N} = \frac{1}{2(N+1) \sin\pa{\frac{\pi}{N+1}} \left( 1+\cos\pa{\frac{\pi}{N+1}} \right)}
\left( 3\sin\pa{\frac{\pi}{N+1}} - \right. \\
- \left. \sin\pa{\frac{3\pi}{N+1}}\right) = \\
= \frac{2}{N+1} \left( 1 - \cos\pa{\frac{\pi}{N+1}} \right) = \frac{4}{N+1} \sin^2\pa{\frac{\pi}{2N+2}}.
\end{gather*}
The formulas for the extremizer coefficients can be simplified using an easily verifiable identity:
\begin{gather*}
(N-j+3) \sin\pa{\frac{\pi j}{N+1}} - (N-j+1) \sin\pa{\frac{\pi (j-2)}{N+1}} + \\
+ (N-j+2) \sin\pa{\frac{\pi (j+1)}{N+1}} - (N-j) \sin\pa{\frac{\pi (j-1)}{N+1}} = \\
= 2 \left( 1+\cos\pa{\frac{\pi}{N+1}} \right) 
\left( (N-j+2) \sin\pa{\frac{\pi j}{N+1}} - \right. \\
\left. - (N-j+1) \sin\pa{\frac{\pi (j-1)}{N+1}} \right).
\end{gather*}

Then
\begin{gather*}
a^0_j = \frac{1}{(N+1) \sin\pa{\frac{\pi}{N+1}}}
\left( (N-j+2) \sin\pa{\frac{\pi j}{N+1}} - \right. \\
\left. - (N-j+1) \sin\pa{\frac{\pi (j-1)}{N+1}} \right), \\
j=1,\ldots,N.
\end{gather*}
Applying the formula \cite[Lemma 2]{DSS} 
$$
U'_{k}(x) = \frac{1}{2(1-x^2)} \left( (k+2)U_{k-1}(x) - kU_{k+1}(x)\right),
$$
where $U_j(x) = U_j (\cos\pa{t}) = \frac{\sin\pa{(j+1)t}}{\sin\pa{t}} = 2^j x^j + \ldots$ is the family of Chebyshev polynomials 
of the second kind, we can express the coefficients $a_j^0$ in terms of the derivatives of Chebyshev polynomials of 
the second kind:
$$
a_j^0 = \frac{U'_{2(N-j+1)} \left( \cos\pa{\frac{\pi}{2N+2}}\right)}{U'_{2N} \left( \cos\pa{\frac{\pi}{2N+2}}\right)},\;
j=1,\ldots,N.
$$
The theorem is fully proved.
\end{proof}

The extremizer allows a closed-form representation.

\begin{theorem}\label{02}
The following representation holds:
\begin{equation}\label{extr}
F^{(0)}(z) = z \frac
{4z^2 \left( z^{2N+2} (1+z^2) + N(1-z^2) + 2\right) \sin^2\pa{\frac{\pi}{2N+2}} + (N+1)(1-z^2)^3}
{(N+1) \left( z^4 - 2z^2 \cos\pa{\frac{\pi}{N+1}} + 1\right)^2}=
\end{equation}
$$
-\frac{z(z^2-1)}{z^4-2z^2\cos\pa{\frac{\pi}{N+1}} + 1}+4\sin^2\frac{\pi}{2N+2} \cdot \frac{z^3(z^2+1)(z^{2N+2}+1)}{(N+1)(z^4-2z^2\cos\pa{\frac\pi{N+1}}+1)^2}.
$$
\end{theorem}

The proof follows from the formulas for the sum of a geometric series and for the derivative of a geometric series.

Note that changing the schlicht normalization $F(0)=0,\, F'(0)=1$ to the normalization $F(0)=0,\,F(1)=1$ leads to other 
extremal polynomials, namely, to the Suffridge and odd Fejér polynomials, which is interesting in itself. Details can
be found in \cite{DHKKS}.

\section{{Extremizer univalence}}

Generally, the task of proving the univalence of functions, and polynomials in particular, is far from simple. The number of 
popular examples of such polynomials is quite limited, especially for odd polynomials. Nevertheless, we managed to prove 
the univalence of the extremal polynomial in our case, which can be considered an undoubted success.

\begin{theorem}\label{03}
The polynomial $F^{(0)}(z)$ is univalent in $\mathbb D.$
\end{theorem}

Denote
$$
u(t):= \text{Re}\left\{ F^{(0)} (e^{it}) \right\} = 
\frac{4\sin^2\pa{\frac{\pi}{2N+2}}}{N+1} \frac{\cos\pa{t} \cos^2\pa{(N+1)t}}{\left( \cos\pa{2t} - \cos\pa{\frac{\pi}{N+1}}\right)^2},
$$
$$
v(t):= \text{Im} \left\{ F^{(0)}(e^{it}) \right\} =
$$
$$
= \frac{\sin\pa{(2N+2)t} \cos\pa{t} \left( 1 - \cos\pa{\frac{\pi}{N+1}} \right) - (N+1) \sin\pa{t} \left( \cos\pa{2t} - \cos\pa{\frac{\pi}{N+1}}\right) }
{(N+1) \left( \cos\pa{2t} - \cos\pa{\frac{\pi}{N+1}}\right)^2}.
$$

We want to show that the set $\Gamma =  \left\{ F^{(0)}(e^{it}),\, t\in\left[ 0, \frac{\pi}{2}\right] \right\} $ defines a simple curve 
(without self-crossing) on the complex plane, which also does not intersect the real axis. That is, first it is to be proved that there holds

{\bf Statement A.} $v(t) \ge 0,\;  t\in\left[ 0, \frac{\pi}{2}\right].$  

Then the curve $\hat{\Gamma} =  \left\{ F^{(0)}(e^{it}),\, t\in\left[ 0, 2\pi \right] \right\} $ turns out to be a simple curve too. 
This follows from the fact that the curve $\hat{\Gamma}$ is symmetric with respect to the imaginary and real axes. I.e., the image of 
the boundary of the disc $D$ is a simple curve on the complex plane; therefore, the mapping will be univalent. 

The curve $\Gamma$ will be simple if the system of equations 
$$
\left\{\begin{array}{ll}
u(t_1)=u(t_2), \\
v(t_1)=v(t_2),
\end{array}\right.
$$
has no solutions in the region $t_1 \in \left[ 0, \frac{\pi}{2}\right],$ $t_2 \in \left[ 0, \frac{\pi}{2}\right],$ $t_1 \neq t_2.$ 
This property will certainly be true if the following statements are fulfilled:

{\bf Statement B.} The function $u(t)$ decreases when $t \in  \left[ 0, \frac\pi{N+1}\right].$

{\bf Statement C.}  For $t\in\left(\frac\pi{N+1},\frac\pi2\right]$ the inequality holds
$$
u(t)<u\left(\frac\pi{N+1}\right).
$$

{\bf Statement D.}
The function $v(t)$ decreases when $t\in\left(\frac\pi{N+1},\frac\pi2\right].$

The proofs of {\it Statements A and B} will be given in Section 3.1, and the proof of {\it Statements C and D} will be given in Section 3.2. In proving {\it Statement B}, we will prove the following stronger statement: the function $u(t)$ is positive and decreasing for $t \in  \pa{0, \frac{3\pi}{2N+2}}.$

\subsection*{\indent 3.1 Proofs of Statements A, B} 
\indent 

{\it Proof of Statement A.} Note that the function $v(t)$ is continuous, except for the point $t_0 = \frac{\pi}{2N+2}$ where the function 
has a removable discontinuity. Let us use the formula 1.391.1 from \cite[p.40]{Grad}
$$
\sin\pa{(2N+2)t} = 2(N+1) \sin\pa{t}\cos\pa{t} \prod\limits_{k=1}^N \left( 1 - \frac{\sin^2\pa{t}}{\sin^2\pa{\frac{k\pi}{2N+2}}} \right)
$$
and well-known trigonometric formulas. Then
$$
v(t) = \frac{\sin\pa{t}}{\sin^2\pa{t} - \sin^2\pa{\frac{\pi}{2N+2}}}
\left( \frac12 - \cos^2\pa{t} \prod\limits_{k=2}^N \left( 1 - \frac{\sin^2\pa{t}}{\sin^2\pa{\frac{k\pi}{2N+2}}} \right) \right).
$$
Each expression in parentheses in the product $\prod\limits_{k=2}^N \left( 1 - \frac{\sin^2\pa{t}}{\sin^2\pa{\frac{k\pi}{2N+2}}} \right) $ is positive 
and decreases when $t \in \left( 0, \frac{\pi}{N+1} \right).$ Hence, the function 
$$
\frac12 - \cos^2\pa{t} \prod\limits_{k=2}^N \left( 1 - \frac{\sin^2\pa{t}}{\sin^2\pa{\frac{k\pi}{2N+2}}} \right)
$$ 
increases on the interval
$\left( 0, \frac{\pi}{N+1} \right)$; moreover, it increases from $-\frac12$ to zero when $t \in \left( 0, \frac{\pi}{2N+2} \right)$ and from zero to 
$\frac12$ when $t \in \left( \frac{\pi}{2N+2}, \frac{\pi}{N+1} \right).$ Thus, the function $v(t)$ is positive on 
$\left( 0, \frac{\pi}{N+1} \right).$

Let $t\in \left( \frac{\pi}{N+1}, \frac{\pi}{2} \right).$ We have 
$$
\sin\pa{(2N+2)t} \cos\pa{t} \left( 1 - \cos\pa{\frac{\pi}{N+1}} \right) - (N+1) \sin\pa{t} \left( \cos\pa{2t} - \cos\pa{\frac{\pi}{N+1}} \right) >
$$
$$
> -2\sin^2\pa{\frac{\pi}{2N+2}} + 2(N+1) \sin\pa{t} \left( \sin^2\pa{t} - \sin^2\pa{\frac{\pi}{2N+2}} \right) \ge
$$
$$
\ge -2\sin^2\pa{\frac{\pi}{2N+2}} + 2(N+1) \sin\pa{\frac{\pi}{N+1}} \left( \sin^2\pa{\frac{\pi}{N+1}} - \sin^2\pa{\frac{\pi}{2N+2}} \right),
$$
because the function $-2\sin^2\pa{\frac{\pi}{2N+2}} + 2(N+1) \sin\pa{t} \left( \sin^2\pa{t} - \sin^2\pa{\frac{\pi}{2N+2}} \right)$ increases on
$\left( \frac{\pi}{N+1}, \frac{\pi}{2} \right).$ The inequality
$$
-2\sin^2\pa{\frac{\pi}{2N+2}} + 2(N+1) \sin\pa{\frac{\pi}{N+1}} \left( \sin^2\pa{\frac{\pi}{N+1}} - \sin^2\pa{\frac{\pi}{2N+2}} \right) > 0
$$
is equivalent to the inequality
$$
(N+1) \sin\pa{\frac{\pi}{N+1}} \left( 1 + 2\cos\pa{\frac{\pi}{N+1}} \right) > 1.
$$
The last inequality is obvious, since $(N+1) \sin\pa{\frac{\pi}{N+1}} > 1.$

{\it Proof of Statement B.} The function $u(t)$ is continuous, except for the point $t_0 = \frac{\pi}{2N+2}$ where the function 
has a removable discontinuity. The formulas 1.391.2 and 1.391.4 from \cite[p.40]{Grad} say
$$
\cos\pa{(N+1)t} = \prod\limits_{k=1}^{\frac{N}{2}} {\left( 1 - \frac{\sin^2\pa{t}}{\sin^2\pa{\frac{(2k-1)\pi}{2N+2}}} \right)}\; 
\text{if } N \text{ is even},
$$
$$
\cos\pa{(N+1)t} = \cos\pa{t} \prod\limits_{k=1}^{\frac{N+1}{2}} {\left( 1 - \frac{\sin^2\pa{t}}{\sin^2\pa{\frac{(2k-1)\pi}{2N+2}}} \right)}\; 
\text{if } N \text{ is odd}.
$$

Then 
$$
\frac{\cos\pa{(N+1)t}}{\cos\pa{2t} - \cos\pa{\frac{\pi}{N+1}}} = \frac{1}{2\sin^2\pa{\frac{\pi}{2N+2}}}
\prod\limits_{k=2}^{\frac{N}{2}} {\left( 1 - \frac{\sin^2\pa{t}}{\sin^2\pa{\frac{(2k-1)\pi}{2N+2}}} \right)}\; 
\text{if } N \text{ is even},
$$
$$
\frac{\cos\pa{(N+1)t}}{\cos\pa{2t} - \cos\pa{\frac{\pi}{N+1}}} = \frac{\cos\pa{t}}{2\sin^2\pa{\frac{\pi}{2N+2}}}
\prod\limits_{k=2}^{\frac{N+1}{2}} {\left( 1 - \frac{\sin^2\pa{t}}{\sin^2\pa{\frac{(2k-1)\pi}{2N+2}}} \right)}\; 
\text{if } N \text{ is odd}.
$$

When $t \in \left( 0, \frac{3\pi}{2N+2} \right),$ each expression in parentheses in the product is positive and decreases
in $t.$ Hence, the entire product decreases, and then the function $u(t)$ also decreases as the product of positive 
decreasing functions. Because $\frac{3\pi}{2N + 2} = \frac{3}{2} \cdot \frac{\pi}{N + 1} > \frac{\pi}{N + 1}$, {\it Statement B} is confirmed.

\subsection*{\indent 3.2 Proof of Statements C and D} 
\indent

\subsubsection{Proof of Statement C}
Let us note that function $\phi(t)=\frac x{(x^2-a^2)^2}$ is increasing for $x\in(0,\abs{a}).$ Therefore,
$$
\frac{\cos\pa{t} \cos^2\pa{(N+1)t}}{\left(\cos\pa{2t}-\cos\pa{\frac\pi{N+1}}\right)^2}\le\frac{\cos\pa{t}}{\left(\cos\pa{2t}-\cos\pa{\frac\pi{N+1}}\right)^2}=\frac{\cos\pa{t}}{4\left(\cos^2\pa{t}-\cos^2\pa{\frac\pi{2(N+1)}}\right)^2}<
$$
$$
\frac{\cos\pa{\frac\pi{N+1}}}{4\left(\cos^2\pa{\frac\pi{N+1}}-\cos^2\pa{\frac\pi{2(N+1)}}\right)^2}=\frac{\cos\pa{\frac\pi{N+1}}}{\left(\cos\pa{\frac{2\pi}{N+1}}-\cos\pa{\frac\pi{N+1}}\right)^2},
$$
which implies the {\it Statement C}.

\subsubsection{Proof of Statement D}
We will begin the proof of {\it Statement D} with calculating the derivative of $v(t)$:
$$
v'(t) = \frac{1}{\sin^2\pa{t} - \sin^2\pa{\frac{\pi}{2N+2}}}
\left( -\frac{\cos\pa{t}}{2} \left( \sin^2\pa{t} + \sin^2\pa{\frac{\pi}{2N+2}} \right) +
\right.
$$
$$
+ \sin^2\pa{\frac{\pi}{2N+2}} \cos\pa{t} \cos\pa{(2N+2)t} -
$$
$$
\left.
- \frac{\sin^2\pa{\frac{\pi}{2N+2}}}{2N+2} \sin\pa{(2N+2)t} \sin\pa{t} 
\left( 1 + \frac{4\cos^2\pa{t}}{\sin^2\pa{t} - \sin^2\pa{\frac{\pi}{2N+2}}} \right)
\right).
$$
We need to check that
\begin{equation}\label{der}
\begin{split}
-\frac{\sin^2\pa{t} + \sin^2\pa{\frac{\pi}{2N+2}}}{2\sin^2\pa{\frac{\pi}{2N+2}}} + \cos\pa{(2N+2)t} - \\
- \frac{\sin\pa{(2N+2)t}}{2N+2}
\left( \tan\pa{t} + \frac{4\cos\pa{t}\sin\pa{t}}{\sin^2\pa{t} - \sin^2\pa{\frac{\pi}{2N+2}}} \right) < 0
\end{split}
\end{equation}
when $t \in \left( \frac{\pi}{N+1}, \frac{\pi}{2} \right),$ $N=2,3,4,\ldots .$

Since
$$
\frac{\sin^2\pa{t} + \sin^2\pa{\frac{\pi}{2N+2}}}{2\sin^2\pa{\frac{\pi}{2N+2}}} > 
\frac{\sin^2\pa{\frac{\pi}{N+1}} + \sin^2\pa{\frac{\pi}{2N+2}}}{2\sin^2\pa{\frac{\pi}{2N+2}}} = 2\cos^2\pa{\frac{\pi}{2N+2}} + 1 \ge 2,
$$
$$
\abs{\cos\pa{(2N+2)t}} \le 1,
$$
then it suffices to prove the inequality
\begin{equation}\label{rinq}
\frac{\sin\pa{(2N+2)t}}{2N+2}
\left( \tan\pa{t} + \frac{4\cos\pa{t}\sin\pa{t}}{\sin^2\pa{t} - \sin^2\pa{\frac{\pi}{2N+2}}} \right) \ge -1.
\end{equation}

Let us split the interval $\left( \frac{\pi}{N+1}, \frac{\pi}{2} \right)$ into two intervals:
$$
\left( \frac{\pi}{N+1}, \frac{\pi}{2} \right) = 
\left( \frac{\pi}{N+1}, \frac{\pi}{2} - \frac{\pi}{2N+2} \right) \cup
\left[ \frac{\pi}{2} - \frac{\pi}{2N+2}, \frac{\pi}{2} \right).
$$

We will consider the behavior of the function
$$
\Psi (t) = \frac{\sin\pa{(2N+2)t}}{2N+2}
\left( \tan\pa{t} + \frac{4\cos\pa{t}\sin\pa{t}}{\sin^2\pa{t} - \sin^2\pa{\frac{\pi}{2N+2}}} \right)
$$
on each of the specified intervals separately.

{\bf \emph{Case 1:}} $t \in I_2 = \left[ \frac{\pi}{2} - \frac{\pi}{2N+2}, \frac{\pi}{2} \right).$

When $N$ is even, the function $\Psi(t)$ is not negative on the interval $I_2$. Therefore, inequality~(\ref{rinq}) is satisfied. 
Let $N$ be odd. Then 
$$
\abs{\frac{\sin\pa{(N+1)t}}{(N+1)\cos\pa{t}}} \le 1,
$$
$$
 \frac{1}{\sin^2\pa{t} - \sin^2\pa{\frac{\pi}{2N+2}}} \le \frac{1}{\cos^2\pa{\frac{\pi}{2N+2}} - \sin^2\pa{\frac{\pi}{2N+2}}} = 
\frac{1}{\cos\pa{\frac{\pi}{N+1}}} \le \sqrt{2},
$$
hence, 
$$
\abs{\Psi(t)} \le \abs{\cos\pa{(N+1)t}} \left( 1 + 4\sqrt{2} \cos^2\pa{t} \right).
$$
It is convenient to make the substitution $\tau = \frac{\pi}{2} - t,$ $\tau \in \left( 0, \frac{\pi}{2N+2} \right].$ Then
$$
\abs{\cos\pa{(N+1)\tau}} \left( 1 + 4\sqrt{2} \sin^2\pa{\tau} \right) \le 
\left( 1 - \frac{(N+1)^2 \tau^2}{2} \right) \left( 1 + 4\sqrt{2} \tau^2 \right) \le 1.
$$
The last inequality is true, since $\frac{2}{(N+1)^2} < \frac{1}{4\sqrt{2}}.$ Thus, $\Psi(t) \ge -1,$ 
$t \in \left[ \frac{\pi}{2} - \frac{\pi}{2N+2}, \frac{\pi}{2} \right).$\\

{\bf \emph{Case 2:}} $t \in I_1 = \left( \frac{\pi}{N+1}, \frac{\pi}{2} - \frac{\pi}{2N+2} \right).$

Note that for $N=2,$ $I_1 = \varnothing;$ and for $N=3,$ $\Psi(t) \ge 0$ when $t \in I_1.$ Therefore, inequality~(\ref{rinq}) 
is valid when $N=2$ or $N=3$.

Let $N \ge 4.$ Estimate the quantity $\displaystyle\frac{\sin\pa{t}}{\sin^2\pa{t} - \sin^2\pa{\pi / (2N+2)}}.$ This function can be represented 
as a sum of two functions which are both decreasing on $I_1$, 
$$
\frac{\sin\pa{t}}{\sin^2\pa{t} - \sin^2\pa{\frac{\pi}{2N+2}}} = \frac{1}{\sin\pa{t} + \sin\pa{\frac{\pi}{2N+2}}} +
\frac{\sin\pa{\frac{\pi}{2N+2}}}{\sin^2\pa{t} - \sin^2\pa{\frac{\pi}{2N+2}}}.
$$
Therefore, it decreases on $I_1.$ Consider the function $\frac{N+1}{2} \left( \cos\pa{t} + \frac12 \frac{\sin^2\pa{t}}{\cos\pa{t}} \right).$
It increases on $I_1,$ since $\left( \cos\pa{t} + \frac12 \frac{\sin^2\pa{t}}{\cos\pa{t}} \right)' = \frac{\sin^3\pa{t}}{2\cos^2\pa{t}}.$ Let us show 
that for $t = \frac{\pi}{N+1}$, the inequality 
\begin{equation}\label{sinq}
\frac{\sin\pa{t}}{\sin^2\pa{t} - \sin^2\pa{\frac{\pi}{2N+2}}} < \frac{N+1}{2} \left( \cos\pa{t} + \frac12 \frac{\sin^2\pa{t}}{\cos\pa{t}} \right)
\end{equation}
holds, which will also imply the validity of inequality~(\ref{sinq}) for all $t \in I_1$.

Let us verify that inequality~(\ref{sinq}) is satisfied for $N=4$ and $N=5$, directly calculating the expressions in both sides 
of this inequality with $t = \frac{\pi}{N+1}.$ We will get $2.35\ldots < 2.55\ldots$ and $2.73\ldots < 3.03\ldots$ 
for $N=4$ and $N=5$, respectively.

Let us show the validity of inequality~(\ref{sinq}) for the rest $N \ge 6.$ Substitute $x = \frac{\pi}{N+1},$ then 
inequality~(\ref{sinq}) for $t = \frac{\pi}{N+1}$ becomes the inequality
\begin{equation}\label{xinq}
\frac{\sin\pa{x}}{\sin^2\pa{x} - \sin^2\pa{\frac{x}{2}}} < \frac{\pi}{2x} \left( \cos\pa{x} + \frac12 \frac{\sin^2\pa{x}}{\cos\pa{x}} \right),\;
x \in \left( 0, \frac{\pi}{7} \right].
\end{equation}
Use the identity 
$$
\sin^2\pa{x} - \sin^2\pa{\frac{x}{2}} = \frac12 \left( 1 + \cos\pa{x} - 2\cos^2\pa{x} \right) = (1 - \cos\pa{x})\left( \frac12 + \cos\pa{x} \right)
$$
and rewrite inequality~(\ref{xinq}) as
\begin{equation}\label{xinq1}
\frac{4}{\pi} x \sin\pa{x}\cos\pa{x} < \left( 1 + \cos^2\pa{x} \right) (1 - \cos\pa{x}) \left( \frac12 + \cos\pa{x} \right).
\end{equation}
Apply the inequality 
\begin{equation}\label{aux}
\frac{4}{\pi} x \sin\pa{x} < 1 - \cos^3\pa{x} = \pa{1 - \cos\pa{x}}\pa{1 + \cos\pa{x} + \cos^2\pa{x}},\;
x \in \left( 0, \frac{\pi}{7} \right]
\end{equation}
to inequality \eqref{xinq1}, thus obtaining 
\begin{gather}\label{xinq2}
\cos\pa{x} \pa{1 - \cos\pa{x}} \left( 1 + \cos\pa{x} + \cos^2\pa{x} \right) \\
< \pa{1 - \cos\pa{x}}\left( 1 + \cos^2\pa{x} \right) \left( \frac12 + \cos\pa{x} \right) \nonumber
\end{gather}
We will give the proof of inequality~(\ref{aux})  at the end of section 3.2, after verifying \eqref{xinq2}.

Verifying \eqref{xinq2} is equivalent to checking the inequality
$$
\left( 1 + \cos^2\pa{x} \right) \left( \frac12 + \cos\pa{x} \right) - \cos\pa{x} \left( 1 + \cos\pa{x} + \cos^2\pa{x} \right) > 0.
$$
We have 
$$
\left( 1 + \cos^2\pa{x} \right) \left( \frac12 + \cos\pa{x} \right) - \cos\pa{x} \left( 1 + \cos\pa{x} + \cos^2\pa{x} \right) = \frac12 \sin^2\pa{x} > 0
$$
when $x \in \left( 0, \frac{\pi}{7} \right].$

Applying \eqref{sinq}, we finally estimate the function $\Psi(t)$.
$$
\abs{\Psi(t)} \le \frac{1}{2N+2} \left( \tan\pa{t} + \frac{4\cos\pa{t}\sin\pa{t}}{\sin^2\pa{t} - \sin^2\pa{\frac{\pi}{2N+2}}} \right) <
$$
$$
< \frac{1}{2N+2} \tan\pa{t} + \cos^2\pa{t} + \frac12 \sin^2\pa{t} = 1 + \frac{1}{2N+2} \tan\pa{t} - \frac12 \sin^2\pa{t} =
$$
$$
= 1 + \frac14 \tan\pa{t} \left( \frac{2}{N+1} - \sin\pa{2t} \right). 
$$
Since $t \in I_1$, this means that $2t \in \pa{\frac{2\pi}{N + 1},\pi - \frac{\pi}{N + 1}}$. Recall that the sine wave increases on this interval for $2t < \frac{\pi}{2}$ and decreases for $2t > \frac{\pi}{2}$. Thus, for $N \geq 6$, using the symmetry of the sine wave we have
$$
\sin\pa{2\pa{\frac{\pi}{2} - \frac{\pi}{2N + 2}}} = \sin\pa{\pi - \frac{\pi}{N + 1}} = \sin\pa{\frac{\pi}{N+1}} < \sin\pa{\frac{2\pi}{N+1}}.
$$
Therefore,
$$
1 + \frac14 \tan\pa{t} \left( \frac{2}{N+1} - \sin\pa{2t} \right) \le 
$$
$$
= 1 + \frac14 \tan\pa{t} \left( \frac{2}{N+1} - \sin\pa{\frac{\pi}{N+1}} \right) < 1,  
$$
since $\sin\pa{\frac{\pi}{N+1}} > \frac{2}{N+1}$ when $N\ge 3.$ Whence, {\it Statement D} is proved.\\

In conclusion, let us prove inequality~(\ref{aux}). Rewrite it in the form
$$
\cos^3\pa{x} < 1 - \frac{4}{\pi} x \sin\pa{x}.
$$
We have 
$$
\cos^3\pa{x} < \left( 1 - \frac12 x^2 + \frac{1}{24} x^4 \right)^3 = 
$$
$$
= 1 - \frac32 x^2 + \frac78 x^4 - \frac14 x^6 + \frac{7}{192} x^8 - \frac{1}{384} x^{10} + \frac{3}{13824} x^{12},
$$
$$
1 - \frac{4}{\pi} x \sin\pa{x} > 1 - \frac{4}{\pi} x^2.
$$
Calculate
$$
1 - \frac{4}{\pi} x^2 - \left( 1 - \frac12 x^2 + \frac{1}{24} x^4 \right)^3 = 
$$
$$
= x^2 \left( \frac{3\pi-8}{2\pi} - \frac78 x^2 \right) + x^6 \left( \frac14 - \frac{7}{192} x^2 \right) +
x^{10} \left( \frac1{384} - \frac{1}{13824} x^2 \right).
$$
This difference is obviously positive when $0 < x^2 \le \frac{4(3\pi-8)}{7\pi},$ i.e. when 
$0 < x \le \sqrt{\frac{4(3\pi-8)}{7\pi}}=0.509\ldots.$ 
But $\frac{\pi}{7} = 0.44\ldots < 0.509,$ therefore, inequality~(\ref{aux}) holds at least for $x \in \left( 0, \frac{\pi}{7} \right].$

The assertion of {\it Theorem~\ref{03}} follows from {\it Statements A, B, C, D}, which we have proved. Thus, {\it Theorem 3} is also 
completely proved.

\section{{Discussion of the results}}
The statement of the problems \eqref{brn} and \eqref{jn} considered above can be generalized as follows.

Optimization problem: on the class of the univalent in $\mathbb D = \left\{ z:\, \abs{z}<1 \right\}$ polynomials with the $T$-fold  symmetry 
of degree $(N-1)T+1$ with real coefficients
$$
F(z) = \sum\limits_{j=1}^{N} {a_j z^{T(j-1)+1}} 
$$
find
\begin{equation}\label{jng}
J^{(T)}_N = \min_{a_j:\,a_1=1} \left( \abs{F \left( e^{\frac{i\pi}{T}} \right)} \right),
\end{equation}
as well as the polynomial (extremizer) that solves the optimization problem.\\

Comparative analysis of polynomials \eqref{extr1} and \eqref{extr2} shows a remarkable feature in the first factor. 
Namely, in the odd case it is the same multiplier, only with odd numbers. Hence, it is logical to assume that in the 
$T$-symmetric case there will be the same factor, only with the numbers forming a $T$-arithmetic sequence.

The middle factor from \eqref{extr1} is absent in \eqref{extr2}, so that we can assume it to equal one. It is logical to assume that in the  $T$-symmetric case this will be the factor that turns into the middle one from formula~\eqref{extr2} 
in the case $T=1,$ and into one in the case $T=2.$ Such a factor has been proposed in \cite{DSST} in the form of the product 
that in the limit as $N\to\infty$ gives the corresponding coefficient of the $T$-symmetric Koebe function. These arguments 
lead to the following hypothesis.\\

\begin{conjecture}\label{con1}
 The extremizer for the problem \eqref{jng} is as follows
$$
F^{(0)}(z) = z+ \sum\limits_{j=2}^{N} {a^{(0)}_j z^{T(j-1)+1}},
$$
$$
a_j^{(0)} = \frac{U'_{T(N-j+1)} \left( \cos\pa{\frac{\pi}{TN+2}}\right)}{U'_{TN} \left( \cos\pa{\frac{\pi}{TN+2}}\right)}
\prod\limits_{k=1}^{j-1} \frac{\sin\pa{\frac{\pi(2+T(k-1))}{TN+2}}}{\sin\pa{\frac{\pi T k}{TN+2}}},\;
j=2,\ldots,N.
$$
Moreover, it is univalent.
\end{conjecture}

In its turn, Conjecture \ref{con1} leads to the following problem. 

\begin{problem}\label{prob}
Determine the asymptotics of the coefficients of $F^{(0)}(z)$ and the quantity $\abs{F^{(0)} \left( e^{\frac{i\pi}{T}} \right)}$.\end{problem}

\begin{figure}[h!]
\centering
\includegraphics[scale=0.2]{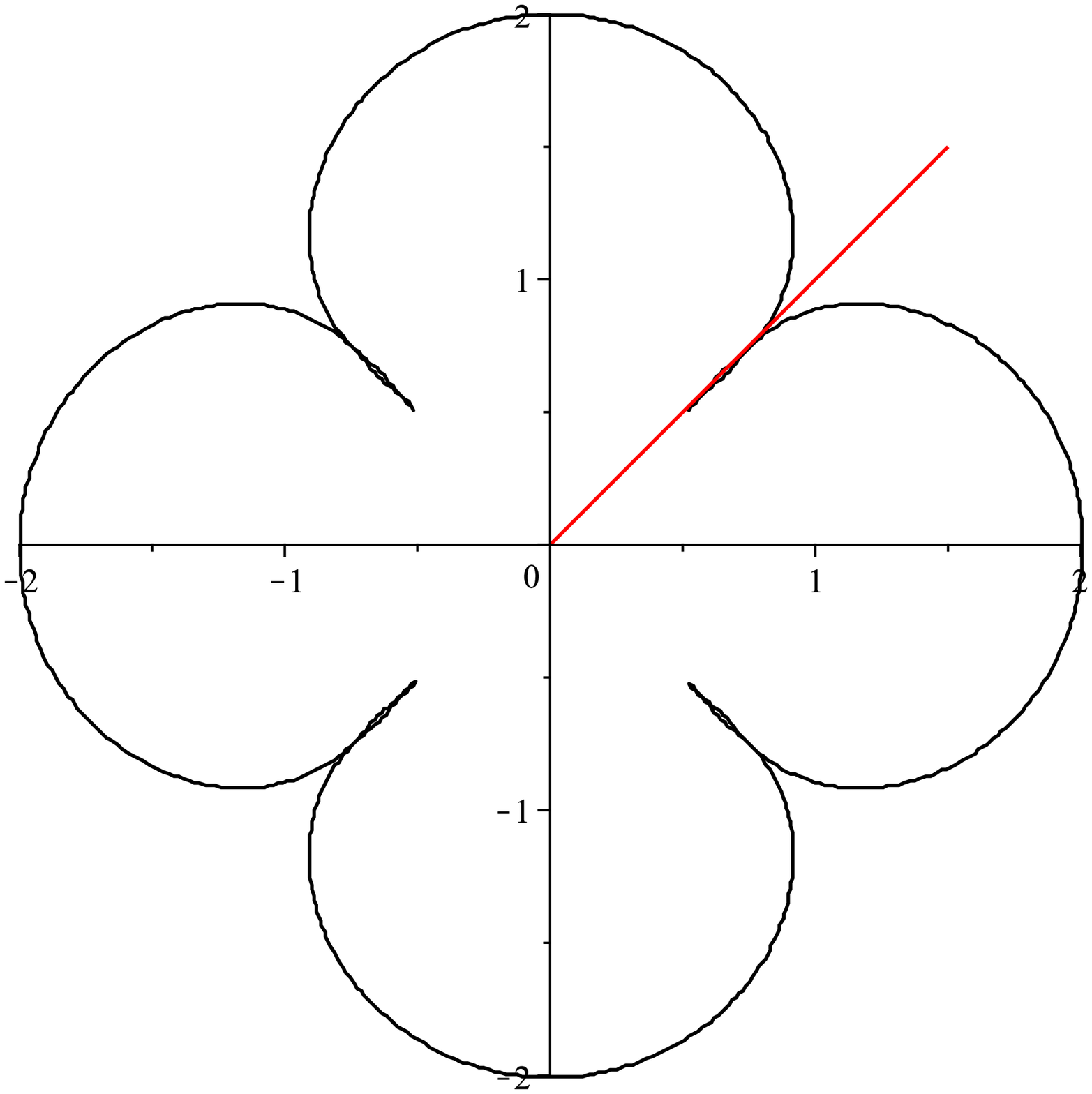}
\hspace{1cm}
\includegraphics[scale=0.2]{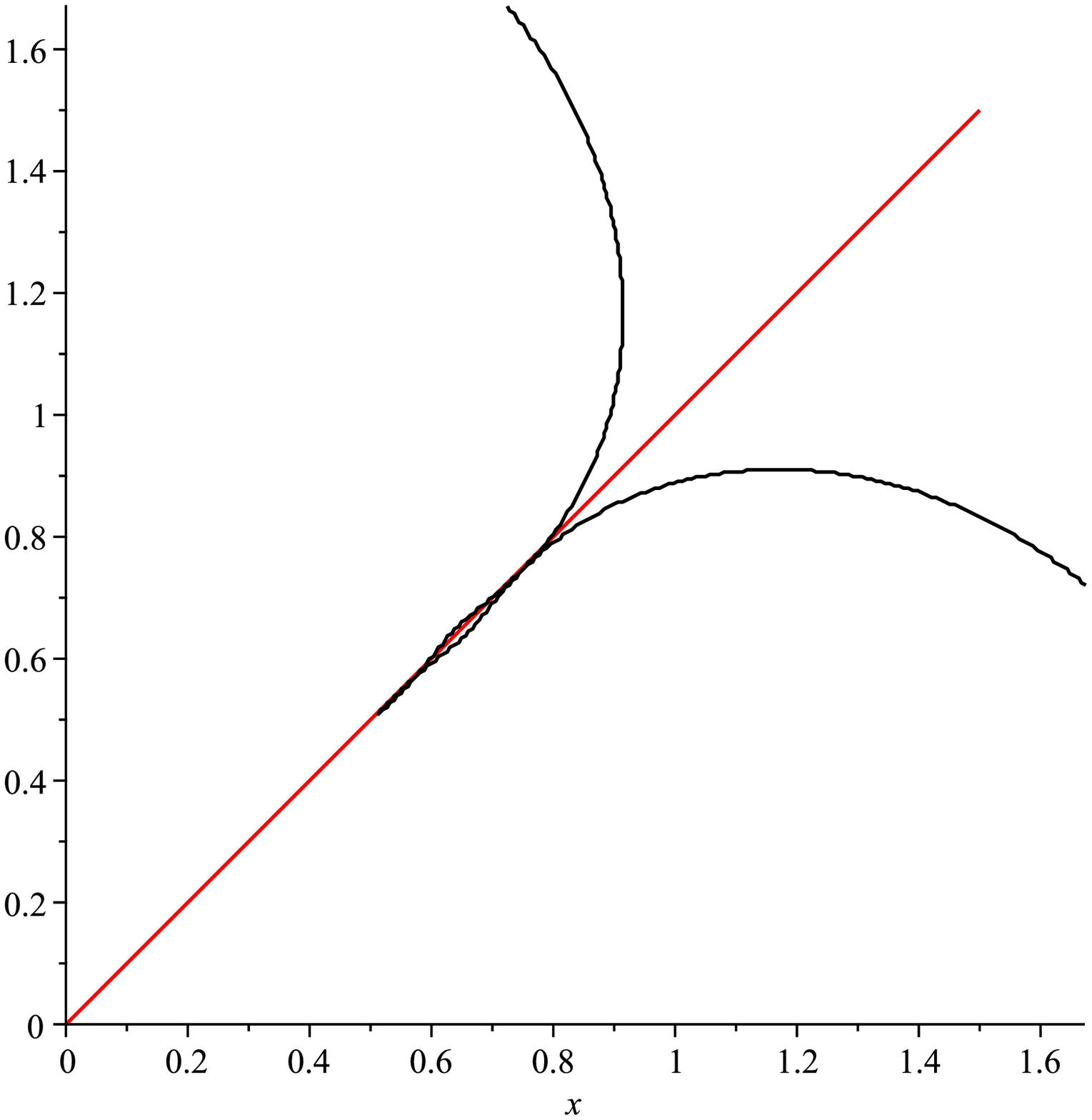}
\caption{The image $F^{(0)}(\mathbb D)$ at $N=7,$ $T=4.$} \label{img}
\end{figure}

Figure~\ref{img} shows the image of the unit disc in the special case $N=7,$ $T=4.$ This picture argues in favor of the 
assumption that the polynomials are univalent, which is confirmed in many other pictures plotted by the authors, but not included in this work.\\


Note that Problem \ref{prob} is a polynomial version of the $T$-symmetric Koebe theorem, which states that the image of the unit disc under the $T$-symmetric schlicht mapping contains a disc of radius $4^{-1/T}.$ The sharpness of the constant is given by the value 
 $\abs{K^{(T)} \left( e^{\frac{i\pi}{T}} \right)},$
where $\displaystyle K^{(T)}=\frac z{(1-z^T)^{2/T}}$ is
$T$-symmetric Koebe function. This leads to the next conjecture.

\begin{conjecture}\label{con2}
 The image of the unit disc under the $T$-symmetric schlicht polynomial degree $N$ contains a disc of radius $\abs{F^{(0)} \left( e^{\frac{i\pi}{T}} \right)}.$ The value is sharp. 
\end{conjecture}

Note that in case $T=1$ the proposed value is $\displaystyle\frac14 \sec^2{\frac{\pi}{N+2}}$ while in case $T=2$ it is $\displaystyle\frac12 \sec^2{\frac{\pi}{2N+2}}.$

\section{Acknowledgment}
The authors would like to thank Larie Ward for her help in preparation of the manuscript.

\bibliographystyle{unsrt}
\bibliography{brandtbib}

\begin{thebibliography}{10}

\bibitem{L-P}
{J.E. Littlewood, R.E.A.C. Paley}.
\newblock A proof that an odd schlicht function has bounded coefficients.
\newblock {\em J. London Math. Soc.}, s1-7 (3):167--169, 1932.

\bibitem{Fekete}
G.~Szeg\"o M.~Fekete.
\newblock {Eine Bemerkung über ungerade schlichte Funktionen}.
\newblock {\em {J. London Math. Soc.}}, 8:85--89, 1933.

\bibitem{Hayman}
{W.K. Hayman}.
\newblock {The asymptotic behaviour of p-valent functions}.
\newblock {\em Proceedings of the London Mathematical Society}, Third Series,
  vol. 5:257--284, 1955.

\bibitem{Brandt1987}
M.~Brandt.
\newblock {\em {Variationsmethoden f\"ur in der Einheitskreisscheibe schlichte
  Polynome}}.
\newblock PhD thesis, Humboldt-Univ., Berlin, 1987.

\bibitem{DSS}
D.~Dmitrishin, A.~Smorodin, and A.~Stokolos.
\newblock An extremal problem for polynomials.
\newblock {\em Appl. Comput. Harmon. Anal.}, 56:283--305, 2022.

\bibitem{DDS}
Stokolos~A. Dmitrishin~D., Dyakonov~K.
\newblock Univalent polynomials and koebe's one-quarter theorem.
\newblock {\em Anal. Math. Phys.}, 9:991--1004, 2019.

\bibitem{Pol-Sz}
{P\'olya G., Szeg\"o G.}
\newblock {\em {Problems and theorems in analysis. II. Theory of functions,
  zeros, polynomials, determinants, number theory, geometry}}.
\newblock {Springer-Verlag}, Berlin, 1998.
\newblock 392 p.

\bibitem{And-Dim}
Roberto Andreani and Dimitar Dimitrov.
\newblock An extremal nonnegative sine polynomial.
\newblock {\em Rocky Mountain Journal of Mathematics}, 33, 07 2001.

\bibitem{Eger}
{E.V. Egerv\'ary, O. Sz\'asz}.
\newblock {Einige Extremalprobleme in Bereiche der trigonometrischen
  Polynomen}.
\newblock {\em Math. Z.}, 27:641--652, 1928.

\bibitem{DHKKS}
{Dmitrishin D., Hagelstein P., Khamitova A., Korenovskyi A., Stokolos A.}
\newblock Fej\'er polynomials and control of nonlinear discrete systems.
\newblock {\em Constr Approx}, 51:383--412, 2020.

\bibitem{Grad}
{I.S. Gradshteyn, I.M. Ryzhik}.
\newblock {\em {Table of Integrals, Series, and Products}}.
\newblock {Academic Press}, Amsterdam, {Eighth} edition, 2014.

\bibitem{DSST}
{Dmitrishin D., Smorodin A., Stokolos A., Tohaneanu M.}
\newblock {Symmetrization of Suffridge polynomials and approximation of
  $T$-symmetric Koebe functions}.
\newblock {\em J. Math. Anal. Appl.}, 503(2), 2020.

\end{thebibliography}

\end{document}